\documentclass[11pt]{amsart}

\usepackage{amssymb,amsmath}
\usepackage{bbm,accents}
\usepackage{a4wide}

\usepackage[colorlinks=true, pdfstartview=FitV, linkcolor=blue, 
citecolor=blue]{hyperref}

\newtheorem{theorem}{Theorem}[section]
\newtheorem{prop}[theorem]{Proposition}
\newtheorem{lemma}[theorem]{Lemma}
\newtheorem{coro}[theorem]{Corollary}
\theoremstyle{definition}

\newcommand{\ts}{\hspace{0.5pt}}
\newcommand{\nts}{\hspace{-0.5pt}}
\newcommand{\CC}{\mathbb{C}}
\newcommand{\RR}{\mathbb{R}\ts}
\newcommand{\NN}{\mathbb{N}}

\newcommand{\dd}{\,\mathrm{d}}
\newcommand{\ee}{\ts\mathrm{e}}
\newcommand{\ii}{\ts\mathrm{i}}
\newcommand{\one}{\mathbbm{1}}
\newcommand{\trans}{{\scriptscriptstyle \mathsf{T}}}
\newcommand{\udo}[1]{\underaccent{$\text{.}$}{#1\ts}\nts}

\newcommand{\diag}{\mathrm{diag}}
\newcommand{\Mat}{\mathrm{Mat}}

\newcommand{\defeq}{\mathrel{\mathop:}=}

\newcommand{\myfrac}[2]{\frac{\raisebox{-2pt}{$#1$}}
  {\raisebox{0.5pt}{$#2$}}}

\begin{document}

\title[Matrix functions of triangular matrices]
{An alternative recursive approach to\\[2mm] 
  functions of simple triangular matrices}

\author{Ellen Baake}
\address{Technische Fakult\"at, Universit\"at Bielefeld, 
         Postfach 100131, 33501 Bielefeld, Germany}

\author{Michael Baake}
\address{Fakult\"at f\"ur Mathematik, Universit\"at Bielefeld, 
         Postfach 100131, 33501 Bielefeld, Germany}

\begin{abstract} 
  The computation of matrix functions is a well-studied problem. Of
  special importance are the exponential and the logarithm of a
  matrix, where the latter also raises existence and uniqueness
  questions. This is particularly relevant in the context of matrix
  semigroups and their generators.  Here, we look at matrix functions
  of triangular matrices, where a recursive approach is possible when
  the matrix has simple spectrum. The special feature is that no
  knowledge of eigenvectors is required, and that the same recursion
  applies to the computation of multiple functions or semigroups
  simultaneously.
\end{abstract}


\maketitle

\section{Introduction}

For any matrix $B\in\Mat(d,\CC)$, its matrix exponential is
given by the convergent series
\begin{equation}\label{eq:exp-series}
  \exp (B) \, = \sum_{n=1}^{\infty} \myfrac{1}{n \ts !} B^{n} .
\end{equation}
There are many ways to determine $\exp (B)$; see \cite[Ch.~10]{Higham}
or \cite{M1, M2} and references therein for thorough expositions. By
an application of the Cayley--Hamilton theorem, it is clear that
$\exp (B)$ is a polynomial in $B$ of degree at most $d-1$. When $B$
has distinct eigenvalues (simple spectrum), or when it is cyclic
(meaning that characteristic and minimal polynomial agree), the degree
is exactly $d-1$, and the coefficients can be computed from the
spectral mapping theorem and the inverse of the Vandermonde matrix.

In practice, several other methods are employed or preferred. The
matrix $B$ is often brought to triangular form, because the
calculation of the exponential of triangular matrices is a standard
problem. Numerically, when the spectrum is simple, it is frequently
treated with the Schur--Parlett algorithm, as described in
\cite[Chs.~9 and 10]{Higham}; see also \cite{HL} for recent
developments.  Here, we present an independent variant that is
motivated by a recursion used in \cite{BBS} in a similar
context. While it has a number of similarities with the Schur--Parlett
algorithm, its initial (or pre-processing) step does \emph{not}
require any input from the function to be computed. Independently of
any numerical implementation, this is of theoretical interest, for
instance for the treatment of matrix semigroups or families of
functions.  Below, we work with lower-triangular matrices and note
that the corresponding results for upper-triangular matrices are
completely analogous.

When a real or complex matrix $A$ has spectral radius
$\varrho^{}_{A} < 1$, the principal matrix logarithm $R$ of $M=\one+A$
is given by the convergent series
\begin{equation}\label{eq:log-series}
  R \, = \, \log (\one+A) \, = \sum_{n=1}^{\infty}
     \frac{(-1)^{n-1}}{n} A^{n} .
\end{equation}
More generally, one can use the integral representation
\cite[Thm.~11.1]{Higham}
\[
   R \, = \int_{0}^{1} \! A \ts (\one + \tau A)^{-1} \dd \tau \ts ,
\]
which is valid whenever no eigenvalue of $A$ lies in
$(-\infty, -1]$.

Below, we revisit the recursive approach to matrix functions of
triangular matrices. The key result is the derivation of a recursion
for a set of coefficients that only depend on the given matrix $B$ and
can then be used for any matrix function of it, provided $B$ has
simple spectrum. The advantage then is that different functions of
$B$, or parameter-dependent ones such as $\exp (t B)$, can be computed
with the same set of coefficients.

The recursive structure is derived in Section~\ref{sec:rec}, which
leads to Proposition~\ref{prop:theta-rec}. In Section~\ref{sec:exp},
this is followed by a more detailed analysis of the matrix exponential
(Theorem~\ref{thm:exp}), where we also take a closer look at the class
of Markov semigroups, and then state the result for general matrix
functions (Corollary~\ref{coro:fun}). Finally, Section~\ref{sec:log}
analyses the matrix logarithm in more detail, aiming at the extraction
of the generator of a Markov semigroup (Theorem~\ref{thm:log} and
Corollary~\ref{coro:rows}).

\section{Recursive structure of triangular matrices}\label{sec:rec}
  
In what follows, empty sums are always zero, or simply understood to
be absent. Further, we mark summation indices by an `underdot' for
clarity.

We begin with the following observation.  Let
$B =(b^{}_{ij})^{}_{1\leqslant i,j \leqslant d}$ be a real or complex
matrix in lower-triangular form, hence with eigenvalues $b^{}_{ii}$.
Assume further that it is diagonalisable,
\[
     B \, = \, U \diag (b^{}_{11}, \ldots , b^{}_{dd} ) \ts U^{-1} ,
\] 
where $U$, and hence also $U^{-1}$, is then lower triangular as well;
compare \cite[Thm.~2.3.1]{HJ}. In fact, the columns of $U$ constitute
a basis of right eigenvectors of $B$, while the rows of $U^{-1}$ are
the corresponding left eigenvectors. Note that $B$ real implies that
$U$ can also be chosen as a real matrix, because the eigenvector
equations can then be solved over the reals. In either case, the
matrix elements of $B$ satisfy
\begin{equation}\label{eq:lin-comb}
    b^{}_{nm} \, = \sum_{m \leqslant \udo{k} \leqslant n} \!
       u^{}_{nk} \ts \tilde{u}^{}_{km} \, b^{}_{kk}   \ts ,
\end{equation}
where the $u^{}_{ij} $ and $\tilde{u}^{}_{ij}$ denote the elements of
$U$ and $U^{-1}$, respectively.  This gives all matrix elements of $B$
as linear combinations of the diagonal ones. While there is a lot of
freedom in the choice of $U$, the linear combinations will be blind to
most of this. Let us therefore define the coefficients
\begin{equation}\label{eq:theta-def}
    \vartheta^{}_{n} (k,m) \, \defeq \, 
       u^{}_{nk} \ts \tilde{u}^{}_{km} \ts ,
\end{equation}
which are well defined once $U$ is fixed. They can only be non-zero
for $m\leqslant k \leqslant n$, due to the triangular structure of $U$
and $U^{-1}$, with $\vartheta^{}_{n} (n,n)=1$. We collect some
properties as follows.

\begin{lemma}\label{lem:orth}
  Let\/ $B=(b^{}_{ij})^{}_{1\leqslant i,j \leqslant d}$ be a real or
  complex matrix in lower-triangular form that is
  diagonalisable. Then, one has
\begin{equation}\label{eq:lem-1}
    b^{}_{nm} \, = \sum_{m\leqslant \udo{k} \leqslant n}
      \! b^{}_{kk} \, \vartheta^{}_{n} (k,m)
\end{equation}
for all\/ $1\leqslant m,n \leqslant d$, with the coefficients from
Eq.~\eqref{eq:theta-def}. They satisfy the two relations
\begin{equation}\label{eq:lem-2}
   \sum_{m\leqslant \udo{\ell} \leqslant n} \! 
   \vartheta^{}_{n} (\ell,m) \, = \, \delta^{}_{nm}   
   \quad \text{and} \quad \!
    \sum_{m\leqslant \udo{i} \leqslant k} \! 
    \vartheta^{}_{n} (k,i) \, \vartheta^{}_{i} (\ell,m)
    \, = \, \delta^{}_{k \ell} \, \vartheta^{}_{n} (k,m) \ts .
\end{equation}
Further, for\/ $k\leqslant m \leqslant n$, one also has the identity
\begin{equation}\label{eq:lem-3}
  \sum_{k\leqslant \udo{\ell} \leqslant n} \! 
  b^{}_{n\ell} \, \vartheta^{}_{\ell} (k,m) \, = \,
  b^{}_{kk}\, \vartheta^{}_{n} (k,m) \ts .
\end{equation}
\end{lemma}

\begin{proof}
  The initial claim and \eqref{eq:lem-1} follow from our previous
  observation in \eqref{eq:lin-comb}.  Next, the first relation in
  \eqref{eq:lem-2} is nothing but $U U^{-1}=\one$, and easily follows
  from Eq.~\eqref{eq:theta-def}, while the second derives from the
  same equation via $U^{-1} U = \one$.
  
  The identity in \eqref{eq:lem-3} now easily follows from replacing
  the $b^{}_{n\ell}$ under the sum by the linear combination from
  \eqref{eq:lem-1}, followed by an application of the second relation
  in \eqref{eq:lem-2}.
\end{proof}

When the spectrum of $B$ is simple, the eigenvectors for each
eigenvalue are unique up to a non-zero multiple. The
$\vartheta$-coefficients are unaffected by this degree of freedom, and
are thus unique. Then, one can solve identity \eqref{eq:lem-3} from
Lemma~\ref{lem:orth} for $\vartheta^{}_{n} (k,m)$ with $k<n$ to obtain
\begin{equation}\label{eq:rec}
    \vartheta^{}_{n} (k,m) \, = \,
    \frac{1}{b^{}_{kk} \nts - b^{}_{nn}}
    \sum_{k\leqslant \udo{\ell} < n} \! b^{}_{n\ell} \,
    \vartheta^{}_{\ell} (k,m) \ts ,
\end{equation}
which, together with the first relation from Lemma~\ref{lem:orth},
permits to determine the $\vartheta$-coefficients \emph{recursively},
without prior (explicit) knowledge of the eigenvectors of $B$.

\begin{prop}\label{prop:theta-rec}
  Let\/ $B$ be a matrix as in Lemma~$\ref{lem:orth}$, and assume that
  it has simple spectrum. Then, the\/ $\ts\vartheta$-coefficients are
  unique, and can be computed recursively via
\[
  \vartheta^{}_{n} (k,m) \, = \, \delta^{}_{nk} \Bigl( \delta^{}_{nm}
  \, - \! \sum_{m\leqslant \udo{\ell} < n} \! \vartheta^{}_{n}
     (\ell,m) \Big) \, + \, \frac{1}{b^{}_{kk}\nts - b^{}_{nn}}
     \sum_{k\leqslant \udo{\ell} < n} \! b^{}_{n\ell} \,
     \vartheta^{}_{\nts\ell} (k,m) \ts ,
\]
where the last sum is understood to be absent for\/ $k=n$.  This
encodes a complete recursion of the\/ $\vartheta^{}_{n} (k,m)$ for
$1\leqslant m \leqslant k \leqslant n$, with the initial conditions\/
$\vartheta^{}_{n} (n,n)=1$ for\/ $1\leqslant n \leqslant d$.
\end{prop}

\begin{proof}
  When $B$ has simple spectrum, the claimed uniqueness follows from
  Eq.~\eqref{eq:theta-def}, because the eigenvectors are unique up to
  an overall factor, which cancels in each term.

  Further, as the eigenvalues are then distinct, we may use
  Eq.~\eqref{eq:rec} and augment it with the relation for $k=n$, taken
  from the first relation in \eqref{eq:lem-2}. This gives the first
  term, while the second is then absent due to our convention that the
  empty sum is zero.

  Since we also have $\vartheta^{}_{n} (n,n) = 1$ for all
  $1\leqslant n \leqslant d$, which serve as the initial conditions,
  it is easy to check that we obtain a complete recursion for the
  $\vartheta$-coefficients.
\end{proof}

Let us point out that the recursive structure will primarily be of
theoretical interest, as its numerical use might suffer from the
standard small denominator problem.

\section{Exponential and other matrix functions}\label{sec:exp}

Proposition~\ref{prop:theta-rec} permits the computation of matrix
functions of $B$ via the diagonal elements and the
$\vartheta$-coefficients. To illustrate this, let us consider
$P(t) = \ee^{t B}$ with $t \in \RR$ and $B$ a real matrix.  Then, we
see that $B$ having simple spectrum implies the relations
\[
     \bigl( P(t) \bigr)_{nm} \, = \sum_{m \leqslant \udo{k} \leqslant n} 
     \! \ee^{t \ts b^{}_{kk}} \, \vartheta^{}_{n} (k,m) 
\]
for all $m,n$ and all $t$. For $t=0$, this simply reflects that
$U U^{-1}=\one$. We thus have the following result, which is formulated
with Markov semigroups in mind.

\begin{theorem}\label{thm:exp}
  Let\/ $B = (b^{}_{ij})^{}_{1\leqslant i,j \leqslant d}$ be a real
  matrix in lower-triangular form, with distinct diagonal entries.
  Then,
  $P(t) = \ee^{t B}= (p^{}_{ij} (t))^{}_{1\leqslant i,j \leqslant d}$
  with\/ $t\in\RR$ is lower triangular and, for\/ $t\ne 0$, has only
  simple eigenvalues. Further, for all\/ $t\in\RR$, one has
\[
  p^{}_{nm} (t) \; = \sum_{m\leqslant \udo{k} \leqslant n} \!
   \ee^{t \ts b^{}_{kk}} \, \vartheta^{}_{n}(k,m)  \, = \,
  \delta^{}_{nm} \ee^{t \ts b^{}_{nn}} \, + \! \sum_{m\leqslant \udo{k} < n}\!
  \vartheta^{}_{n} (k,m) \bigl( \ee^{t \ts b^{}_{kk}} 
    - \ee^{t \ts b^{}_{nn}}\bigr)
\]
with the recursively{\ts}-defined coefficients from
Proposition~$\ref{prop:theta-rec}$.
\end{theorem}

\begin{proof}
  The lower-triangular form of $P(t)$ is clear from the exponential
  series and the fact that the lower-triangular nature is preserved
  under sums and products of matrices, as well as under taking
  limits. When $t \ne 0$, all eigenvalues of $P(t)$, which are the
  $p^{}_{nn} (t) = \ee^{- t \ts b^{}_{nn}}$, are simple because the
  exponential function is a bijection between $\RR$ and $\RR_{+}$.
  
  Since $B$ and $P(t)$ commute for all $t\in \RR$, they share the same
  set of eigenvectors, and hence the same, unique
  $\vartheta$-coefficients. The claimed formula then follows via
  Eq.~\eqref{eq:lem-2} in Lemma~\ref{lem:orth}, where the coefficients
  can be computed via the recursion from
  Proposition~\ref{prop:theta-rec}.
\end{proof}

When the matrix $B$ has the additional property that all row sums are
$0$, we know that $(1,1,\ldots , 1)^{\trans}$ is a right eigenvector
of $B$ with eigenvalue $0$, and hence also an eigenvector of every
$P(t)=\ee^{t B}$, then always with eigenvalue $1$, which we summarise
as follows.

\begin{coro}\label{coro:row-sum}
  When, under the condition of Theorem~\textnormal{\ref{thm:exp}}, the
  matrix\/ $B$ has zero row sums, its first row must be zero, and\/
  $P(t) = \ee^{t B}$ has row sums\/ $1$, for every\/ $t\in\RR$. \qed
\end{coro}

Let us quickly look at the recursion from a different angle, as known
from the Schur--Parlett algorithm; compare \cite[Sec.~9.2]{Higham}.
When $P(t)=\ee^{t B}$, we clearly have $P(t) B = B \ts P(t)$ for all
$t$.  Suppressing the notation for the explicit time dependence of
$P$, this implies
\begin{equation}\label{eq:high}
  p^{}_{nm} (b^{}_{nn} \nts - b^{}_{mm}) \, = \,
  b^{}_{nm} (p^{}_{nn} \nts - p^{}_{mm}) \, + \!
  \sum_{m < \udo{k} < n} \! (b^{}_{nk}\ts p^{}_{km}
    - p^{}_{nk}\ts b^{}_{km})
\end{equation}
for all $1\leqslant m \leqslant n \leqslant d$, together with
$p^{}_{ij} = b^{}_{ij} =0$ for all $i < j$.  Clearly, \eqref{eq:high}
gives no condition for $n=m$, where we have
$p^{}_{nn} (t) = \exp(t \ts b^{}_{nn})$. When $m=n-1$, we see the
condition
\begin{equation}\label{eq:high-2}
    p^{}_{n,n-1}  \, = \,
    b^{}_{n,n-1} \frac{p^{}_{n,n} \nts - p^{}_{n-1,n-1}}
    {b^{}_{n,n} \nts - b^{}_{n-1,n-1}}\ts .
\end{equation}
This is well defined, and can be inverted, because neither the
numerator nor the denominator of the fraction vanishes under our
conditions. Note that this agrees with the recursion from
Proposition~\ref{prop:theta-rec} for the elements in the first
subdiagonal.  In comparison, the recursion for the
$\vartheta$-coefficients primarily works row by row, where
$\vartheta^{}_{n}$ specifies how to add the $n\ts $th row, while the
iteration based on Eq.~\eqref{eq:high} progresses from one subdiagonal
to the next. One advantage of the $\vartheta$-coefficients is that
they only use the (implicit) information from the eigenvectors of $B$,
but nothing from $f(B)$. Unfortunately, this alternative does not seem
to favourable on the numerical side,\footnote{We thank Thorsten
  H\"{u}ls for some numerical experiments and comparisons with
  standard \textsc{MatLab} routines.}  which seems related to the ill
conditioning of the (underlying) diagonalisation matrix; compare
\cite[Sec.~4.5]{Higham}.

Let us nevertheless derive two consequences of this alternative,
recursive structure. Recalling that $b^{}_{kk} \ne b^{}_{\ell\ell}$
for all $k\ne \ell$, a close inspection of Eq.~\eqref{eq:high} reveals
that this recursion can be iterated until all $p^{}_{nm}$ with $n>m$
are expressed as linear combinations of the $p^{}_{kk}$, with
coefficients that are functions of the entries of $B$ only and hence
time independent. Since this works for every $t>0$, and since the
functions $\ee^{t \ts\ts b^{}_{nn}}$ are linearly independent, we get
the following consequence.

\begin{coro}\label{coro:exp}
  The recursion from Proposition~$\ref{prop:theta-rec}$ and the
  recursion implicitly defined by Eq.~\eqref{eq:high}, together with\/
  $p^{}_{nn} (t) = \ee^{t \ts b^{}_{\nts nn}}$ for\/
  $1\leqslant n \leqslant d$, give the same result, that is, the same
  expression of each matrix element\/ $p^{}_{nm} (t)$ as a linear
  combination of the diagonal elements of\/ $P(t)$, with
  time-independent coefficients. \qed
\end{coro}

Let us extend Theorem~\ref{thm:exp} to its complex counterpart, in
some generality. Recall that the matrix function of a diagonalisable
matrix $B$ can be defined for a given function $f$ when all
eigenvalues of $B$ lie in the domain of $f$. One then says that $f$ is
defined on the spectrum of $B$; compare \cite[Def.~1.2]{Higham}. The
statement is as follows.

\begin{coro}\label{coro:fun}
  Let\/ $B$ be a complex matrix in lower triangular form, with
  distinct entries on the diagonal, so the spectrum of\/ $B$ is simple
  and reads\/ $\sigma (B) = \{ b^{}_{11}, \ldots , b^{}_{dd}
  \}$. Let\/ $f$ be a function defined on\/ $\sigma (B)$. Then, the
  matrix\/ $f(B)$ is lower triangular and given by
\[
    \bigl( f(B) \bigr)_{nm} \, = \sum_{m\leqslant \udo{k} \leqslant n}
    f (b^{}_{kk}) \, \vartheta^{}_{n} (k,m) \, = \,
    \delta^{}_{nm} \ts f(b^{}_{nn})
    \; + \! \sum_{m\leqslant \udo{k} < n} \! \vartheta^{}_{n} (k,m)
    \bigl(  f(b^{}_{kk}) - f(b^{}_{nn}) \bigr) ,
\]   
with the recursively{\ts}-defined coefficients from
Proposition~$\ref{prop:theta-rec}$. Given\/ $B$, the\/
$\vartheta$-coefficients are then the same for all functions\/ $f$.
\end{coro}

\begin{proof}
  Under the assumptions, the matrix $B$ is diagonalisable.  The
  lower-triangular form of $f(B)$ is a consequence of the explicit
  diagonalisation procedure, where the $\vartheta$-coefficients are
  unique and recursively determined by
  Proposition~\ref{prop:theta-rec}, and clearly independent of $f$.
\end{proof}

One difference to Theorem~\ref{thm:exp} is that, in general, the
spectrum of $f(B)$ need not be simple. In principle,
Corollary~\ref{coro:fun} applies to the matrix logarithm as
well. Since this is a particularly important function, we look at it
in some more detail.

\section{The principal logarithm of triangular
   matrices}\label{sec:log}

Let us turn to the principal logarithm of suitable lower-triangular
matrices. While we have Markov matrices in mind, the formulation can
be given for a slightly more general situation. First, let $P$ be a
real matrix with positive, simple eigenvalues only. In particular, $P$
is invertible and diagonalisable. Then, by Culver's theorem
\cite{Culver}, $P$ has a \emph{unique} real logarithm, which means
that the equation $P = \ee^B$ with $B$ real has precisely one
solution.

Here, we are interested in the case that $P$ is also lower triangular,
which is motivated by the occurrence of such structures in recently
studied processes from mathematical genetics \cite{fast}.  Then, the
diagonal elements of $P$ are positive and distinct.  If all
eigenvalues of $P=\one+A$ lie in $(0,1]$, as they do when $P$ is
Markov (under the above assumptions), the real logarithm is unique and
given by the convergent series from Eq.~\eqref{eq:log-series}, which
can be calculated in various ways. Once $B$ is determined, one can
consider the group $\{ P(t)= \ee^{t B}: t\in\RR\}$ of matrices, where
$P(t)^{-1} = P(-t)$ holds for every $t \in \RR$. The question now is
whether $B$ can be extracted from $P(t)$ with an arbitrary (but fixed)
$t\ne 0$ in a recursive way that resembles Theorem~\ref{thm:exp}. The
answer is affirmative, also already for semigroups, and can be stated
as follows.

\begin{theorem}\label{thm:log}
  Let\/ $\{ P(t) :t \geqslant 0\}$ be a continuous semigroup of real
  $d{\ts\times}d$-matrices that satisfy\/ \mbox{$P(0)=\one$} and\/
  $P(t+s) = P(t) P(s)$ for\/ $t,s\geqslant 0$.  Assume further that,
  for some\/ $t^{}_{0} > 0$, the matrix\/ $P(t^{}_{0})$ is lower
  triangular, with distinct diagonal elements and thus simple
  spectrum.  Then, there is a unique real matrix\/
  $B=(b^{}_{ij})^{}_{1\leqslant i,j \leqslant d}$ such that\/
  $P(t) = \ee^{t B}$ for all\/ $t\geqslant 0$. This matrix is lower
  triangular and has simple spectrum as well.
  
  Further, for any\/ $0 \ne t\in\RR$, one has 
\[
  b^{}_{nm}  \; = \sum_{m\leqslant \udo{k} \leqslant n} \!
  b^{}_{kk} \, \eta^{}_{n}(k,m) \, = \, 
  \delta^{}_{nm} b^{}_{nn}\, + \! \sum_{m\leqslant \udo{k} < n}\!
  \eta^{}_{n} (k,m) \bigl( b^{}_{kk} \nts - b^{}_{nn} \bigr)
\]
with\/ $b^{}_{nn} =  t^{-1} \log\ts \bigl(p^{}_{nn} (t )\bigr)$ for\/
$1\leqslant n \leqslant d$ and the coefficients
\[
  \eta^{}_{n} (k,m) \, = \, \delta^{}_{nk} \Bigl( \delta^{}_{nm}
   \, - \! \sum_{m\leqslant \udo{\ell} < n} \! \eta^{}_{n} (\ell,m) \Big)
   \, + \! \sum_{k\leqslant \udo{\ell} < n}
   \frac{p^{}_{n\ell}(t) \, \eta^{}_{\nts\ell} (k,m) }
        {p^{}_{kk} (t) - p^{}_{nn} (t)}  \ts ,
\]
where the last sum is absent when\/ $k=n$.  This encodes a complete
recursion for the\/ $\eta^{}_{n} (k,m)$ for
$1\leqslant m \leqslant k \leqslant n$, with the initial conditions\/
$\eta^{}_{n} (n,n)=1$ for\/ $1\leqslant n \leqslant d$.  Moreover,
these\/ $\eta$-coefficients are time independent and agree with the\/
$\vartheta$-coefficients from Eq.~\eqref{eq:theta-def}.
\end{theorem}

\begin{proof}
  The first claim is a consequence of \cite[Thm.~2.9]{EN} in
  conjunction with Culver's theorem \cite{Culver} on the existence of
  a real logarithm. Indeed, the semigroup property (already with
  continuity only at $0$) guarantees the existence of a (generally
  complex) matrix $R$ such that $P(t) = \ee^{t R}$ holds for all
  $t\geqslant 0$. Since $R$ commutes with $P(t^{}_{0})$, which is
  simple and lower triangular, $R$ must be lower triangular as well
  \cite[Thm.~1.29]{Higham}.  This implies that $P(t)$ is lower
  triangular for all $t\geqslant 0$, via the power series of
  $\ee^{t R}$.

  Under our additional assumptions, we know that $P(t^{}_{0})$ has a
  unique real logarithm, $B^{}_{0}$ say, so
  $P(t^{}_{0}) = \ee^{B^{}_{0}}$. Lower triangularity of $B^{}_{0}$ is
  inherited from $P(t^{}_{0})$, by the same argument used for $R$, and
  all eigenvalues of $B_0$ are then real.  They are also simple, as a
  consequence of the spectral mapping theorem and the bijectivity of
  the real logarithm as a mapping from $\RR_{+}$ to $\RR$.  Clearly,
  $P( m \ts\ts t^{}_{0}) = \ee^{m B^{}_{0}}$ for any $m\in\NN$.  Now,
  consider $P(t^{}_{0} / n)$ for $n\in\NN$. This matrix is real and
  still has simple, positive eigenvalues, as the occurrence of $-1$
  can be excluded by also considering
  $P\bigl( \frac{t^{}_{0}}{2n}\bigr)$, which cannot have eigenvalues
  $\ii$ or $-\ii$.  Repeating the argument with the unique real
  logarithm shows that $P(t^{}_{0} / n) = \ee^{B^{}_{1}}$, where
  $B^{}_{0} = n B^{}_{1}$ by the previous argument. This leads to
  $P(q \ts\ts t^{}_{0}) = \ee^{q B^{}_{0}}$ for all non-negative
  rational numbers $q$, and thus to $P(t) = \ee^{t B}$ with
  $B=B^{}_{0}/t^{}_{0}$ by continuity. Consequently, we also get $R=B$
  by considering the (existing) derivative of $P(t)$ at $0$, and the
  claimed uniqueness is clear.

  To compute $B$ from any $P=P(t)$ for some $t>0$, we use our previous
  argument twice.  Write
  $P = U \diag (p^{}_{11}, \ldots, p^{}_{dd}) \ts U^{-1}$, where the
  columns of $U$ contain the right eigenvectors of $P(t^{}_{0})$, each
  of which is unique up to a non-zero overall constant. If we now
  define the $\eta\ts$-coefficients in complete analogy to
  Eq.~\eqref{eq:theta-def}, we get the relations
\[
    p^{}_{nm} \, = \sum_{m\leqslant \udo{k} \leqslant n} \!
    p^{}_{kk} \, \eta^{}_{n} (k,m)
\]  
together with the $\eta$-recursions as stated. Due to our assumptions,
$P(t^{}_{0})$ and $B$ share the same set of eigenvectors, and we can
use the $\eta$-coefficients also to calculate the $b^{}_{nm}$ from the
$b^{}_{kk}$, where we know the latter as
$b^{}_{kk} = t^{-1} \log (p^{}_{kk})$.
  
Now, we know that $B$ has simple spectrum, so our original argument
gives the relation among the elements of $B$ as stated in
Lemma~\ref{lem:orth}, with the $\vartheta$-coefficients from
Eq.~\eqref{eq:theta-def}. Since the eigenvectors are the same, each
being unique up to a non-zero multiplicative constant, which does
\emph{not} affect the value of the coefficients, we get the equality
of the $\eta\ts$- and the $\vartheta$-coefficients as claimed.
\end{proof} 
  
A special case is relevant and in use in probability theory, in
particular in the context of Markov embedding \cite{King,BS1}.  When
the matrix $P(t^{}_{0})$ has all row sums equal to $1$, it has
$(1,1,\ldots ,1)^{\trans}$ as an eigenvector with eigenvalue $1$, and
$B$ inherits this eigenvector, now with eigenvalue $0$, so all row
sums of $B$ are $0$. Then, by Theorem~\ref{thm:exp}, all $P(t)$ have
row sums $1$, and we have the following.

\begin{coro}\label{coro:rows}
  If, under the general assumptions of
  Theorem~\textnormal{\ref{thm:log}}, $P(t^{}_{0})$ has unit row sums,
  this holds for all\/ $P(t)$, and the real matrix\/ $B$ has zero row
  sums.  \qed
\end{coro}

More generally, when we start from a simple matrix $P$, which may be
real or complex, one can use Corollary~\ref{coro:fun} for the
logarithm as well. The subtle point here is that one has to select an
appropriate branch of the (complex) logarithm, which impacts the
outcome. This is a subtle issue, and we refer to \cite[Ch.~11]{Higham}
for a detailed exposition.

\section*{Acknowledgements}

It is our pleasure to thank Wolf-J\"{u}rgen Beyn and Thorsten H\"{u}ls
for helpful discussions, comments and numerical experiments.  This
work was supported by the German Research Foundation (DFG, Deutsche
Forschnungsgemeinschaft), within the CRC 1283 at Bielefeld University
(project ID 317210226).


\smallskip
\begin{thebibliography}{99}
\small

\bibitem{fast}
E.~Baake and M.~Baake,
Haldane linearisation done right: Solving the nonlinear
recombination equation the easy way,
\textit{Discr.\ Cont.\ Dynam.\ Syst.\ A}
\textbf{36} (2016) 6645--6656;
\texttt{arXiv:1606.05175}.

\bibitem{BBS}
E.~Baake, M.~Baake and M.~Salamat, The general recombination
equation in continuous time and its solution,
\textit{Discr.\ Cont.\ Dynam.\ Syst.\ A} \textbf{36} (2016)
63--95 and 2365--2366 (erratum and addendum);
\texttt{arXiv:1409.1378}.

\bibitem{BS1}
M.~Baake and J.~Sumner,
Notes on Markov embedding 
\textit{Lin.\ Alg.\ Appl.} \textbf{594} (2020) 262--299;
\texttt{arXiv:1903.08736}.

\bibitem{Culver}
W.J.~Culver,
On the existence and uniqueness of the real logarithm
of a matrix, \textit{Proc.\ Amer.\ Math.\ Soc.} 
\textbf{17} (1966) 1146--1151.

\bibitem{EN}
K.-J.~Engel and R.~Nagel,
\textit{One-Parameter Semigroups for Linear Evolution Equations},
Springer, New York (2000).

\bibitem{Higham}
N.J.~Higham, 
\textit{Functions of Matrices: Theory and Computation},
SIAM, Philadelphia, PA (2008).

\bibitem{HL}
N.J.~Higham and X.~Liu,
A multi-precision derivative-free Schur--Parlett algorithm for
computing matrix functions,
\textit{SIAM J.\ Matrix Anal.\ Appl.} \textbf{42} (2021) 1401--1422.

\bibitem{HJ}
R.A.~Horn and C.R.~Johnson, 
\textit{Matrix Analysis}, 2nd ed., 
Cambridge University Press, Cambridge (2013).

\bibitem{King}
J.F.C.~Kingman,
The imbedding problem for finite Markov chains,
\textit{Z.\ Wahrscheinlichkeitsth.\ verw.\ Geb.}
\textbf{1} (1962) 14--24.

\bibitem{M1}
C.B.~Moler and C.F.~Van Loan,
Nineteen dubious ways to compute the exponential of a matrix,
\textit{SIAM Rev.} \textbf{20} (1978) 801--836.

\bibitem{M2}
C.B.~Moler and C.F.~Van Loan,
Nineteen dubious ways to compute the exponential of a matrix,
twenty-five years later,
\textit{SIAM Rev.} \textbf{45} (2003) 3--49.

\end{thebibliography}
\end{document}